\documentclass[12pt,bezier]{article}
\usepackage{times}
\usepackage{booktabs}
\usepackage{pifont}
\usepackage{floatrow}
\usepackage{caption}
\usepackage[fleqn]{amsmath}
\usepackage{amsfonts,amsthm,amssymb,mathrsfs,bbding}
\usepackage{txfonts}
\usepackage{graphics,multicol}
\usepackage{graphicx}
\usepackage{color}
\usepackage{caption}
\usepackage{cite}
\usepackage{latexsym,bm}
\usepackage{indentfirst}
\usepackage[colorlinks=true,anchorcolor=blue,filecolor=blue,linkcolor=blue,urlcolor=blue,citecolor=blue]{hyperref}
\usepackage{extarrows}
\usepackage{mathtools}
\evensidemargin=\oddsidemargin\topmargin=-1.5cm

\newtheorem{thm}{Theorem}[section]
\newtheorem{prob}{Problem}[section]
\newtheorem{claim}{Claim}

\newtheorem{lem}{Lemma}[section]

\theoremstyle{definition}

\addtocounter{section}{0}

\begin{document}
\title{Extremal problems on $[a, b]$-covered graphs\footnote{Supported by the National Natural Science Foundation of China
{(No. 12371361)} and Distinguished Youth Foundation of Henan Province {(No. 242300421045)}.}}
\author{{\bf Qixuan Yuan}, {\bf Ruifang Liu}\thanks{Corresponding author.
E-mail addresses: rfliu@zzu.edu.cn (R. Liu), yuanqixuan6@gmail.com (Q. Yuan), yuanjj@zzu.edu.cn (J. Yuan).}, {\bf Jinjiang Yuan}\\
{\footnotesize School of Mathematics and Statistics, Zhengzhou University, Zhengzhou, Henan 450001, China}}

\date{}

\maketitle
{\flushleft\large\bf Abstract}
A spanning subgraph $F$ is called an $[a,b]$-factor of $G$ if $a\leq d_F(v)\leq b$ for any vertex $v\in V(G).$
A graph $G$ is $[a,b]$-covered if for each edge $e$ of $G$ there is an $[a,b]$-factor containing it.
From a structural perspective, the existence of an $[a,b]$-factor is a necessary condition for a graph to be $[a,b]$-covered, making the set of $[a,b]$-factor-free graphs a subset of non-$[a,b]$-covered graphs.

For $a=b=1$, an $[a,b]$-covered graph is a matching covered graph.
The structural theory of matching covered graphs constitutes a cornerstone of modern matching theory. Determining whether a given graph is matching covered is a fundamental problem in structural graph theory. Lucchesi et al. [SIAM J. Discrete Math., 2018] showed that a connected graph $G$ is matching covered if and only if every barrier of $G$ is a stable set. In this paper, we completely characterize the extremal graphs that maximize the size or the spectral radius among all non-matching-covered graphs.

For $a \leq b$ and $b \geq 2,$ Hao and Li [Electron. J. Combin., 2024] investigated the extremal problems on $[a,b]$-factor graphs: If $G$ contains no $[a,b]$-factors, then $e(G)\leq \binom{n-1}{2}+a-1$ with equality if and only if $G\cong H_{n,a},$ where $H_{n,a} = K_{a-1} \vee (K_{n-a} \cup K_1).$ Moreover, if $G$ contains no $[a,b]$-factors, then $\rho(G)\leq \rho(H_{n,a})$ with equality if and only if $G \cong H_{n,a}.$ Judging from the structral characterization, non-$[a,b]$-covered graphs exhibit highly complex structures, making the associated extremal problems significantly challenging. To overcome this, we develop a novel minimum-degree forcing technique. Combining this technique and spectral-structural analysis, we in this paper provide complete characterizations of the extremal graphs that maximize the size or the spectral radius within the set of non-$[a,b]$-covered graphs. An intriguing phenomenon revealed by our results is that $H_{n,a}$ remains both the size-extremal graph and the spectral extremal graph for this larger set of non-$[a,b]$-covered graphs. Consequently, our results strengthen the results of Hao-Li.

\begin{flushleft}
\textbf{Keywords:}
Matching covered graphs, $[a, b]$-covered graphs, Size, Spectral radius
\end{flushleft}
\textbf{AMS Classification:} 05C50; 05C35

\section{Introduction}

Let $G$ be a simple and undirected graph with vertex set $V(G)$ and edge set $E(G).$ The {\it order} of $G$ is defined as $v(G)=|V(G)|,$ and its {\it size} is $e(G)=|E(G)|.$ Denote by $N_G(v)$ the set of vertices adjacent to $v$ in $G$ with $v \in V(G),$ and $|N_G(v)|$ is said to be the {\it degree} of $v$ in $G$, denoted by $d_G(v).$ We denote by $\delta(G)$ the {\it minimum degree} of $G.$ For a vertex subset $S \subseteq V(G),$ let $G[S]$ be the subgraph of $G$ induced by $S.$ For $V_1,V_2\subset V(G),$ $E(V_1,V_2)=\{v_1v_2:v_1\in V_1,v_2\in V_2 \mbox{ and } v_1v_2\in E(G)\}$ and $e(V_1,V_2)=|E(V_1,V_2)|.$ Let $G_1$ and $G_2$ be two vertex-disjoint graphs. Denote by $G_1 \cup G_2$ the disjoint union of $G_1$ and $G_2.$ The {\it join} of two disjoint graphs $G_1$ and $G_2$, denoted by $G_1 \vee G_2,$ has vertex set $V(G_1) \cup V(G_2)$ and edge set $E(G_1)\cup E(G_2)\cup \{uv:u\in V(G_1), v\in V(G_2)\}.$ The {\it complement} $\overline{G}$ of $G$ is the graph whose vertex set is $V(G)$ and whose edges are the pairs of nonadjacent vertices of $G.$ Let $A(G)$ be the adjacency matrix of $G.$ The largest eigenvalue of the adjacency matrix $A(G),$ denoted by $\rho(G),$ is called the {\it spectral radius} of $G.$

Given a graph $H$, a graph is said to be {\it $H$-free}, if it does not contain a subgraph isomorphic to $H.$ Extremal graph theory focuses on identifying the global thresholds that guarantee the presence of specific local structures. A classic problem in extremal graph theory is the Tur\'{a}n-type problem, which asks for the maximum number of edges in an $H$-free graph of order $n$ and the characterization of the corresponding extremal graphs. The study of the Tur\'{a}n-type problems dates back to Mantel's theorem \cite{Mantel1907}, which asserts that any triangle-free graph $G$ of order $n$ has at most $\big\lfloor \frac{n^{2}}{4}\big\rfloor$ edges. In 1941, Tur\'{a}n \cite{Turan1941} extended Mantel's theorem and showed that $e(G)\leq e(T_{n,k})$ for every $K_{k+1}$-free graph $G$ of order $n\geq k+1$ and the $k$-partite Tur\'{a}n graph $T_{n,k}$ is the unique extremal graph. This classic problem attract significant research interest and has led to a wealth of results.

Following the development of classical extremal graph theory, spectral extremal graph theory has emerged as a young but highly active field. One of the most fundamental problems in this area is to determine the maximum spectral radius of a given graph family and to characterize the corresponding extremal graphs. In 1986, Wilf \cite{Wilf1986} initially showed that $\rho(G)\leq n(1-\frac{1}{k})$ for every $K_{k+1}$-free graph $G$ of order $n$. This result was later sharpened independently by Nikiforov \cite{Nikiforov2007} and Guiduli \cite{Guiduli1996}. Nikiforov \cite{Nikiforov2007} proved that $\rho(G)\leq \rho(T_{n,k})$ for every $K_{k+1}$-free graph $G$ of order $n,$ with equality if and only if $G\cong T_{n,k}$, which is known as the spectral Tur\'{a}n theorem. In 2010, Nikiforov \cite{Nikiforov2010} formally proposed a spectral version of the Tur\'{a}n-type problem: What is the maximum spectral radius of an $H$-free graph of order $n$? In recent decades, a wide range of results has emerged focusing on this problem.

Further investigations have revealed a consistency between the size extremal graphs and the spectral extremal graphs of various given graph families. Wang, Kang and Xue \cite{Wang2023} confirmed a conjecture by Cioab\u{a}, Desai and Tait \cite{Cioaba2022} in a stronger form. They proved that if the maximum size of $H$-free graphs of order $n$ is $e(T_{n,r})+O(1),$ then the set of graphs that maximize the spectral radius is a subset of the set of graphs that maximize the number of edges. Motivated by these developments, Liu and Ning \cite{Liu2023} proposed an interesting problem of characterizing all graphs $H$ for which every spectral extremal graph is also a size extremal graph for sufficiently large $n$. An extensive body of literature has emerged concerning this problem. For example, see \cite{Turan1941,Nikiforov2007,Guiduli1996} for $K_{k+1},$ \cite{Erdos1995,Cioaba2020,Zhai2022} for $K_1\vee kK_2,$ and \cite{Wang2024} for blow-up of star forest.

Let $g,f$ be two integer-valued functions defined on $V(G).$ A spanning subgraph $F$ is called a {\it $(g,f)$-factor} of $G$ if $g(v)\leq d_F(v)\leq f(v)$ for any vertex $v\in V(G).$ Let $a$ and $b$ be two positive integers with $a\leq b.$ A $(g,f)$-factor is called an {\it $[a,b]$-factor} if $g(v)\equiv a$ and $f(v)\equiv b$ for any $v\in V(G).$ In 1970, Lov\'{a}sz \cite{Lovasz1970} generalized the conclusion of Tutte's $f$-factor theorem in \cite{Tutte1952} to $(g, f)$-factors.

\begin{thm} [\!\!\cite{Lovasz1970}]
A graph $G$ has a $(g, f)$-factor if and only if for any two disjoint subsets $S,$ $T$ of $V(G),$
\begin{eqnarray*}
f(S)-g(T)+\sum_{x\in T}d_{G-S}(x)-\hat{q}_G(S, T) \geq 0,
\end{eqnarray*}
where $\hat{q}_G(S, T)$ denotes the number of components $C$ in $G-S-T$ such that $g(v) = f(v)$ for all $v \in V(C)$ and $f(V(C)) + e_G(V(C), T)\equiv 1 \pmod{2}.$
\end{thm}

Denote by $H_{n,\gamma}$ an $n$-vertex graph consisting of an $(n-1)$-clique together with an additional vertex that is connected to exactly $(\gamma-1)$ vertices of the clique, that is, $H_{n,\gamma} = K_{\gamma-1} \vee (K_{n-\gamma} \cup K_1).$ Let $\mathcal{G}_e$ be the set of $n$-vertex graphs with maximum size that do not contain an $[a, b]$-factor. Hao and Li \cite{Hao2024} determined the maximum size of graphs forbidding $[a, b]$-factors and characterized all graphs in $\mathcal{G}_e$.

\begin{thm} [\!\!\cite{Hao2024}]\label{th1.2}
Let $a\leq b$ be two positive integers, and $G$ be a graph of order $n$, where
$n\geq a+1$ and $na\equiv 0\pmod{2}$ when $a=b.$ If $G$ contains no $[a, b]$-factors, then $e(G) \leq \binom{n-1}{2}+a-1$ with equality if and only if one of the following holds:\\
(i) $G\cong H_{n,a}$ or $K_{1,3},$ if $ab = 1$ or $ab = 2;$\\
(ii) $G\cong H_{n,a}$ or $K_2\vee 3K_1,$ if $a=b=2;$\\
(iii) $G\cong H_{n,a},$ if $b\geq 3.$
\end{thm}

Moreover, let $\mathcal{G}_{sp}$ be the set of $n$-vertex graphs with maximum spectral radius that do not contain an $[a, b]$-factor. Hao and Li \cite{Hao2024} determined the maximum spectral radius of an $n$-vertex graph forbidding $[a, b]$-factors and characterized all graphs in $\mathcal{G}_{sp}$, which strengthens the result of Wei and Zhang \cite{Wei2023}.

\begin{thm} [\!\!\cite{Hao2024}]\label{th1.3}
Let $a\leq b$ be two positive integers, and $G$ be a graph of order $n$, where $n\geq a+1$ and $na\equiv 0\pmod{2}$ when $a = b$. If $G$ contains no $[a, b]$-factors, then $\rho(G)\leq \rho(H_{n,a})$ with equality if and only if $G \cong H_{n,a}.$
\end{thm}

By Theorems \ref{th1.2} and \ref{th1.3}, $\mathcal{G}_{sp}\subseteq \mathcal{G}_e.$ This result provides a positive instance of the problem of Liu and Ning \cite{Liu2023}. In \cite{Little1974}, Little introduced the concept of a {\it matching covered graph}, which is a graph such that for every edge $e$ there exists a $1$-factor containing it. Matching covered is an important and interesting concept in graph theory. Many significant properties of matching covered graphs have been extensively studied, such as ear decomposition \cite{Kothari2015, de2014}, nonfeasible sets \cite{Liu2020} and others \cite{de1996,de2002,He2019}. Lucchesi et al. \cite{Lucchesi2018} presented a tight sufficient condition to guarantee that a graph is matching covered. Szigeti \cite{Szigeti2001} proposed some structural results for extensions of matching-covered graphs. In \cite{Liu1988}, Liu generalized matching covered graphs to $(g,f)$-covered graph.
A graph $G$ is {\it $(g,f)$-covered} if for each edge $e$ of $G$ there is a $(g,f)$-factor containing it.
If $g(v) \equiv a$ and $f(v) \equiv b$ for all $v \in V(G),$ then a $(g,f)$-covered graph is called an {\it $[a,b]$-covered graph}. In \cite{Liu1988}, Liu gave a necessary and sufficient condition for a graph to be $[a,b]$-covered.

For any $S,T \subseteq V(G)$ and $S\cap T = \emptyset,$
let $C$ be a component of $G-S-T.$
If $a = b,$ then we say that $C$ is {\it odd} or {\it even} according to $e_{G}(T, V(C))+b|V(C)|$ being odd or even.
If $a \neq b,$ then we say that $C$ is {\it neutral}.
Let $o_G(S,T)$ denote the number of odd components of $G-S-T.$
Define $\varepsilon(S,T)$ as follows:\\
(i) $\varepsilon(S,T) = 2,$ if one of the following conditions holds: (1) $S$ is not independent.
(2) There is an even component $C$ of $G-S-T$ such that there is an edge between $V(C)$ and $S$ or there is a cut edge $e$ of $C$ such that $e_{G}(T,V(C_i))+b|V(C_i)| \equiv 0 \pmod{2}$ for $i=1,2,$ where $C_1$ and $C_2$ are the components of $C-e.$\\
(ii) $\varepsilon(S,T) = 1,$ if neither of (1) and (2) holds and there is a neutral component of $G-S-T$ such that there is an edge joining $C$ and $S.$\\
(iii) $\varepsilon(S,T) = 0,$ otherwise.

\begin{thm}[\!\cite{Liu1988}]\label{th1.4}
Let $G$ be a graph, and let $a,$ $b$ be integers with $1\leq a \leq b.$ Then $G$ is $[a,b]$-covered if and only if for all $S,T \subseteq V(G),$ $S\cap T = \emptyset,$
\begin{eqnarray*}
\theta_G(S,T) = b|S|-a|T|+\sum_{x\in T} d_{G-S}(x)-o_G(S,T) \geq \varepsilon(S,T).
\end{eqnarray*}
\end{thm}

Let $\mathcal{G}_e^*$ be the set of maximum size $n$-vertex graphs that are not $[a, b]$-covered, and let $\mathcal{G}_{sp}^*$ be the set of maximum spectral radius $n$-vertex graphs that are not $[a, b]$-covered. Note that if a graph is $[a, b]$-covered, then it must contain an $[a, b]$-factor, but the converse is not necessarily true. Motivated by Theorems \ref{th1.2} and \ref{th1.3}, a natural and interesting problem arises:
\begin{prob}\label{pro1}
Does the inclusion $\mathcal{G}_{sp}^* \subseteq \mathcal{G}_e^*$ hold?
\end{prob}

Utilizing Theorem \ref{th1.4}, we completely characterize the extremal graphs that maximize the size within the set of non-$[a,b]$-covered graphs.

\begin{thm}\label{main0}
Let $a$ and $b$ be two positive integers, and let $G$ be a graph of order $n$ and $na\equiv 0\pmod{2}$ when $a=b.$\\
(i) For $a \leq b,$ $b \geq 2$ and $n \geq 3a+4,$ if $G$ is not $[a,b]$-covered, then $e(G) \leq \binom{n-1}{2}+a-1$ with equality if and only if $G\cong H_{n,a}.$\\
(ii) For $a=b=1$ and $n \geq 4,$ if $G$ is not matching-covered, then $e(G) \leq \binom{n-1}{2}+2$ with equality if and only if $G \cong H_{n,3}$ or $K_3\vee 3K_1.$
\end{thm}

In this paper, we also completely characterize the extremal graphs that maximize the spectral radius  among all non-$[a,b]$-covered graphs, which strengthens Theorem \ref{th1.3}.

\begin{thm}\label{main1}
Let $a$ and $b$ be two positive integers, and let $G$ be a graph of order $n$ and $na\equiv 0\pmod{2}$ when $a=b.$\\
(i) For $a \leq b,$ $b \geq 2$ and $n \geq 3a+4,$ if $G$ is not $[a,b]$-covered, then $\rho(G) \leq \rho(H_{n,a})$ with equality if and only if $G\cong H_{n,a}.$\\
(ii) For $a=b=1$ and $n \geq 4,$ if $G$ is not matching-covered, then $\rho(G) \leq \rho(H_{n,3})$ with equality if and only if $G\cong H_{n,3}.$
\end{thm}

Based on Theorems \ref{main0} and \ref{main1}, we completely confirm Problem \ref{pro1}.

\section{Preliminaries}

In this section, we introduce some necessary results, which are essential to the proofs of our main theorem.

\begin{lem}[\!\cite{Hong2001,Nikiforov2002}]\label{lem2.1}
Let $G$ be a graph with minimum degree $\delta$. Then
\begin{eqnarray*}
\rho(G)\leq \frac{\delta-1}{2}+\sqrt{2e(G)-\delta n+\frac{(\delta+1)^2}{4}}.
\end{eqnarray*}
\end{lem}

\begin{lem}\label{lem2.2}
Let $G$ be a graph with minimum degree $\delta$. If $\delta\geq a$ and $\rho(G) \geq n-2,$ then $e(\overline{G})\leq n-\lceil\frac{a}{2}\rceil-1.$
\end{lem}

\begin{proof}
By Lemma \ref{lem2.1}, we have
\begin{eqnarray*}
n-2 \leq \rho(G) \leq \frac{\delta-1}{2}+\sqrt{2e(G)-\delta n+\frac{(\delta+1)^2}{4}}.
\end{eqnarray*}
Hence $e(G)\geq\binom{n-1}{2}+\lceil\frac{\delta}{2}\rceil.$
Since $\delta \geq a$, we have $e(G)\geq\binom{n-1}{2}+\lceil\frac{a}{2}\rceil.$ Therefore, $e(\overline{G})\leq n-\lceil\frac{a}{2}\rceil-1.$
\end{proof}

\begin{lem}\label{lem2.3}
Let $l_1\geq l_2 \geq \cdots \geq l_q$ be positive integers, where $q\geq 1.$
Then $e(K_s\vee (K_{l_1}\cup K_{l_2}\cup \cdots \cup K_{l_q})) \leq e(K_s \vee (K_{n-s-q+1}\cup (q-1)K_1)).$
\end{lem}
\begin{proof}
First, for each $i \in \{2, \ldots, q\},$ we select a vertex $u_i \in K_{l_i}.$ Next we delete $l_i-1$ edges between $u_i$ and $V(K_{l_i})\setminus \{u_i\}$ and add all $l_1({l_i}-1)$ edges between $V(K_{l_i})\setminus \{u_i\}$ and $V(K_{l_1}).$
By repeating this operation for each $i \in \{2, \ldots, q\}$, we obtain $K_s \vee (K_{n-s-q+1}\cup (q-1)K_1).$
Since $l_1\geq l_2 \geq \cdots \geq l_q,$ we have $l_1({l_i}-1)\geq l_i-1.$
Hence $e(K_s\vee (K_{l_1}\cup K_{l_2}\cup \cdots \cup K_{l_q})) \leq e(K_s \vee (K_{n-s-q+1}\cup (q-1)K_1)).$
\end{proof}

\section{Proofs}

In this section, we give the proofs of Theorems \ref{main0} and \ref{main1}. Before presenting our proofs, we first introduce some necessary lemmas.

\begin{lem}\label{lem3.1}
Let $1 \leq a \leq b$ be two integers, and let $G$ be a graph of order $n\geq 3a+4.$ Suppose that $G$ is not $[a,b]$-covered, that is, there exist two subsets $S,T \subseteq V(G)$ with $S\cap T = \emptyset$ such that
\begin{eqnarray}\label{eq1}
\theta_G(S,T) = b|S|-a|T|+\sum_{x\in T} d_{G-S}(x)-o_G(S,T) < \varepsilon(S,T)\leq 2.
\end{eqnarray}
If $b\geq 2,$ $e(\overline{G}) \leq n-\lceil\frac{a}{2}\rceil-1$ and $|S|\leq |T|+1,$ then $|T|\leq 4.$\\
\end{lem}

\begin{proof}
Denote $s = |S|,$ $t = |T|,$ $T'=G-S-T,$ $t' = |V(T')|$ and $o=o_G(S,T).$
Then
\begin{eqnarray}\label{eq2}
\sum_{x\in T} d_{G-S}(x) &=& 2e_{G-S}(T)+e_{G-S}(T,T')\nonumber\\
&\geq& (n-s-1)t-2e(\overline{G}).
\end{eqnarray}
Moreover, by (\ref{eq1}) and $n\geq s+t+o$, we obtain that
\begin{eqnarray}\label{eq3}
\sum_{x\in T} d_{G-S}(x) &\leq& at-bs+o+1\nonumber\\
&\leq& at-bs+(n-s-t)+1.
\end{eqnarray}
Combining (\ref{eq2}) and (\ref{eq3}), we have
\begin{eqnarray}\label{eq4}
(n-s-1)t-2e(\overline{G}) \leq at-bs+(n-s-t)+1.
\end{eqnarray}
By $s\leq t+1,$ $s\leq \frac{n+1}{2}.$ Since $n\geq 3a+4,$ it follows that $n-s-a-1\geq \frac{n-1}{2}-a-1>0.$ Note that $e(\overline{G}) \leq n-\lceil\frac{a}{2}\rceil-1\leq n-\frac{a}{2}-1.$ Combining (\ref{eq4}), $b\geq 2,$ $s\leq t+1$ and $n\geq 3a+4,$ we have
\begin{eqnarray*}
t &\leq& 3+\frac{(-b+2)s-t+2a+2}{n-s-a-1}\\
&\leq& 3+\frac{-s+2a+3}{n-s-a-1}\leq 4,
\end{eqnarray*}
as desired.
\end{proof}

\begin{lem}\label{lem3.2}
Let $a < b$ be two positive integers, and let $S,T \subseteq V(G)$ and $S \cap T = \emptyset.$
Then $|S| \geq \varepsilon(S,T).$
\end{lem}
\begin{proof}
Since $a<b,$ all components of $T'$ are neutral.
If $|S|=0,$ then $S$ is independent and there is no edge between $S$ and the components of $T',$ which implies that $\varepsilon(S,T)=0=|S|.$
If $|S|=1,$ then $S$ is an independent set. Hence $\varepsilon(S,T) \leq 1=|S|.$
If $|S|\geq 2,$ then $|S|\geq 2 \geq \varepsilon(S,T).$
Hence we obtain that $|S|\geq \varepsilon(S,T).$
\end{proof}

Before proving our main theorems, we establish the following crucial lemma by applying the minimum-degree forcing technique, which effectively forces the graph to be $[a,b]$-covered under the upper bound of $e(\overline{G})$.

\begin{lem}\label{lem3.3}
Let $1 \leq a \leq b$ be two integers and $G$ be a graph of order $n$ with $\delta(G)\geq a,$ $n \geq 3a+4$ and $na\equiv 0\pmod{2}$ if $a=b.$\\
(i) If $b\geq 2$ and $e(\overline{G})\leq n-\lceil\frac{a}{2}\rceil-1,$ then $G$ is $[a,b]$-covered.\\
(ii) If $a=1,$ $b\geq 2$ and $e(\overline{G})\leq n-1,$ then $G$ is $[a,b]$-covered.
\end{lem}
\begin{proof}
(i) Suppose to the contrary that $G$ is not $[a,b]$-covered.
We distinguish our proof into two cases.

\vspace{1.5mm}
\noindent\textbf{Case 1.} $1 \leq a < b.$
\vspace{1.5mm}

By Theorem \ref{th1.4}, there exist two subsets $S,T \subseteq V(G)$ with $S\cap T = \emptyset$ such that
\begin{eqnarray}\label{eq5}
\theta_G(S,T) = b|S|-a|T|+\sum_{x\in T} d_{G-S}(x) < \varepsilon(S,T),
\end{eqnarray}
where
\begin{align*}
\begin{split}
\varepsilon(S,T)= \left \{
\begin{array}{ll}
    2,                    & S \mbox{ is not independent},\\
    1,                    & \mbox{$S$ is independent and there is a component $C$ of $G-S-T$ such}\\
                          & \mbox{that there exists an edge joining $C$ and $S,$}\\
    0,                    & \mbox{otherwise}.
\end{array}
\right.
\end{split}
\end{align*}
Let $s = |S|$ and $t = |T|.$
Since $\delta(G) \geq a,$ $\delta(G-S) \geq a-s.$
If $t \leq b-1,$ by $1 \leq a < b$ and Lemma \ref{lem3.2}, then
\begin{eqnarray*}
\theta_G(S,T) &=& bs-at+\sum_{x\in T} d_{G-S}(x)\\
&\geq& bs-at+t(a-s)\\
&=& (b-t)s\\
&\geq& s \geq \varepsilon(S,T),
\end{eqnarray*}
which contradicts (\ref{eq5}). Hence $t \geq b\geq 2.$

\begin{claim}\label{cla1}
$s \leq t-1 \leq 3.$
\end{claim}
\begin{proof}
Suppose that $s \geq t.$ Combining $b\geq a+1$ and Lemma \ref{lem3.2}, we have
\begin{eqnarray*}
\theta_G(S,T)\geq bs-at \geq (a+1)s-at \geq s \geq \varepsilon(S,T),
\end{eqnarray*}
which contradicts (\ref{eq5}).
Hence $s\leq t-1.$
Recall that $n\geq 3a+4,$ $e(\overline{G}) \leq n-\lceil\frac{a}{2}\rceil-1$ and $G$ is not $[a,b]$-covered. Combining $s\leq t-1,$ $b\geq 2$ and Lemma \ref{lem3.1}, we have $t \leq 4.$ Hence $s \leq t-1 \leq 3.$
\end{proof}
Define $T'=G-S-T$ and $t'=|V(T')|.$ Since $e(\overline{G})\leq n-\lceil\frac{a}{2}\rceil-1,$ $2e(\overline{G}[T])+e_{\overline{G}}(T,V(T')) \leq t(t-1)+\Big(n-\frac{t(t-1)}{2}-\lceil\frac{a}{2}\rceil-1\Big).$
Hence we have
\begin{eqnarray}\label{eq6}
\sum_{x\in T}d_{G-S}(x) &\geq& t(t+t'-1)-\Big(t(t-1)+\Big(n-\frac{t(t-1)}{2}-\Big\lceil\frac{a}{2}\Big\rceil-1\Big)\Big)\nonumber\\
&=& tt'-\Big(n-\frac{t(t-1)}{2}-\Big\lceil\frac{a}{2}\Big\rceil-1\Big).
\end{eqnarray}
By (\ref{eq6}), we have
\begin{eqnarray*}
\theta_G(S,T)&=& bs-at+\sum_{x\in T} d_{G-S}(x)\\
&\geq& (a+1)s-at+tt'-\Big(n-\frac{t(t-1)}{2}-\Big\lceil\frac{a}{2}\Big\rceil-1\Big)\\
&=& (a+1)s-at+t(n-s-t)-\Big(n-\frac{t(t-1)}{2}-\Big\lceil\frac{a}{2}\Big\rceil-1\Big)\\
&=& -\frac{t^2}{2} + \Big(n-s-a-\frac{1}{2}\Big)t+(a+1)s-n+\Big\lceil\frac{a}{2}\Big\rceil+1.
\end{eqnarray*}
Let $f(t)=-\frac{t^2}{2} + \left(n-s-a-\frac{1}{2}\right)t+(a+1)s-n+\lceil\frac{a}{2}\rceil+1.$ Combining $t\geq b\geq 2$ and Claim \ref{cla1}, we have $2\leq t\leq 4.$
Hence $f(t) \geq \min\{f(2),f(4)\}.$
It follows from $n\geq 3a+4$ and Claim \ref{cla1} that $f(2)=n+(a-1)s-2a+\lceil\frac{a}{2}\rceil-2>2$ and $f(4)=3n+(a-3)s-4a+\lceil\frac{a}{2}\rceil-9\geq 5a+\lceil\frac{a}{2}\rceil-3> 2.$
Hence $f(t)> 2\geq \varepsilon(S,T),$ which contradicts (\ref{eq5}).

\vspace{1.5mm}
\noindent\textbf{Case 2.} $2 \leq a = b.$
\vspace{1.5mm}

By Theorem \ref{th1.4}, there exist two subsets $S,T \subseteq V(G)$ with $S\cap T = \emptyset$ such that
\begin{eqnarray*}
\theta_G(S,T) = as-at+\sum_{x\in T} d_{G-S}(x)-o_G(S,T) < \varepsilon(S,T),
\end{eqnarray*}
where
\begin{align*}
\begin{split}
\varepsilon(S,T)= \left \{
\begin{array}{ll}
    2,                    & \mbox{$S$ is not independent, or}\\
                          & \mbox{there is an even component $C$ of $G-S-T$ such that there exists}\\
                          & \mbox{an edge joining $C$ and $S$ or there exists a cut edge $e$ of $C$ such that }\\
                          & \mbox{$e_G(T,V(C_i))+a|V(C_i)| \equiv 0\mbox{ (mod }2)$ for $i= 1,2,$ where $C_1$ and $C_2$ are}\\
                          & \mbox{the components of $C-e,$}\\
    0,                    & \mbox{otherwise}.
\end{array}
\right.
\end{split}
\end{align*}

Next we divide the following proof into three cases.

\vspace{1.5mm}
\noindent\textbf{Case 2.1.} $t = 0.$
\vspace{1.5mm}

In this case, we have
\begin{eqnarray}\label{eq7}
\theta_G(S,T) = as-o_G(S,T)< \varepsilon(S,T).
\end{eqnarray}

\vspace{1.5mm}
\noindent\textbf{Case 2.1.1.} $s = 0.$
\vspace{1.5mm}

Since $e(\overline{G})\leq n-\lceil\frac{a}{2}\rceil-1,$ $G$ is connected. Note that $e(T, V(G))+b|V(G)|=na\equiv 0 \pmod{2}.$ Then $G-S-T$ is an even component, and hence $\theta_G(S,T)=-o_G(S,T) = 0.$
Note that $S=\emptyset$. If there is no cut edge in $G,$ then $\varepsilon(S,T)=0,$ and hence $\theta_G(S,T) =\varepsilon(S,T),$ which contradicts (\ref{eq7}).
Hence there exists a cut edge $e$ in $G.$
Let $C_1$ and $C_2$ are two components of $G-e.$
Without loss of generality, let $|V(C_1)|\leq|V(C_2)|.$
We claim that $|V(C_1)| = 1$ and $|V(C_2)| = n-1.$ Otherwise, $e(\overline{G})\geq |V(C_1)||V(C_2)|-1\geq 2(n-2)-1 > n-2\geq n-\lceil\frac{a}{2}\rceil-1,$ which contradicts $e(\overline{G})\leq n-\lceil\frac{a}{2}\rceil-1.$
Then $G\cong H_{n,2},$ and hence $\delta(G)=1,$ which contradicts $\delta(G)\geq a.$

\vspace{1.5mm}
\noindent\textbf{Case 2.1.2.} $s = 1.$
\vspace{1.5mm}

We first prove the following claim.
\begin{claim}\label{cla2}
$G-S$ is a subgraph of $K_{n-2}\cup K_1.$
\end{claim}
\begin{proof}
Suppose that $G-S$ has only one component $C.$
If $C$ is even, then $\theta_G(S,T) = as-o_G(S,T) =a \geq 2 \geq \varepsilon(S,T),$ contradicting (\ref{eq7}).
Note that $S$ is independent.
If $C$ is odd, then $\varepsilon(S,T) = 0,$ and hence $\theta_G(S,T) = as-o_G(S,T) \geq a-1 \geq 1 > \varepsilon(S,T),$ which contradicts (\ref{eq7}).
So $G-S$ has at least two components.
Suppose to the contrary that $G-S$ is not a subgraph of $K_{n-2}\cup K_1.$
By $n\geq 3a+4,$ we have
\begin{eqnarray*}
e(\overline{G}) \geq 2(n-3)> n-\Big\lceil\frac{a}{2}\Big\rceil-1,
\end{eqnarray*}
a contradiction.
\end{proof}
By Claim \ref{cla2}, $G$ is a subgraph of $H_{n,2}.$ Hence $\delta(G)\leq 1,$ which contradicts $\delta(G)\geq a.$

\vspace{1.5mm}
\noindent\textbf{Case 2.1.3.} $s\geq 2.$
\vspace{1.5mm}

Let $q$ be the number of components of $T'.$
By (\ref{eq7}), $as-q\leq as - o_G(S,T)\leq 1.$ Then $q \geq as - 1.$
Let $l_1\geq l_2 \geq \cdots \geq l_q$ be the number of vertices of $q$ components of $T'.$
Define $G' = K_s\vee (K_{l_1}\cup K_{l_2}\cup \cdots \cup K_{l_q}).$
Then $G$ is a spanning subgraph of $G',$ and hence $e(G) \leq e(G').$
According to Lemma \ref{lem2.3}, we have $e(G') \leq e(K_s \vee (K_{n-s-(q-1)}\cup (q-1)K_1)).$
By $q \geq as - 1,$ $s\leq \frac{1}{a}(q+1).$
Combining $q\geq as-1,$ $n\geq s+q$ and $s\leq \frac{1}{a}(q+1),$ we have
\begin{eqnarray*}
e(\overline{G})-\Big(n-\Big\lceil\frac{a}{2}\Big\rceil-1\Big)&\geq& e(\overline{G'})-\Big(n-\Big\lceil\frac{a}{2}\Big\rceil-1\Big)\\
&\geq&(q-1)(n-s-(q-1))+\frac{(q-1)(q-2)}{2}-\Big(n-\frac{a}{2}-1\Big)\\
&=& (q-2)n-\frac{q^2}{2}-\Big(s-\frac{1}{2}\Big)q+s+\frac{a}{2}+1\\
&\geq& \frac{q^2}{2}-\frac{3}{2}q-s+\frac{a}{2}+1\\
&\geq& \frac{q^2}{2}-\Big(\frac{3}{2}+\frac{1}{a}\Big)q-\frac{1}{a}+\frac{a}{2}+1.
\end{eqnarray*}
Let $f(q)=\frac{q^2}{2}-(\frac{3}{2}+\frac{1}{a})q-\frac{1}{a}+\frac{a}{2}+1.$
Then the axis of symmetry of $f(q)$ is $q=\frac{1}{a}+\frac{3}{2}\leq 2,$ which implies that $f(q)$ is increasing for $q \geq 3.$
If $q \geq 4$, then $f(q)\geq f(4) = \frac{a}{2} - \frac{5}{a} + 3.$ Since $a \geq 2$, $f(4)>0.$ Then $e(\overline{G})>n-\lceil\frac{a}{2}\rceil-1,$ a contradiction.
Therefore, we only need to consider $q=3$. Combining $as-1\leq q=3,$ $a\geq 2$ and $s\geq 2,$ we have $a=2$ and $s=2.$ By $n\geq 3a+4=10,$ we have
\begin{eqnarray*}
e(\overline{G})-\Big(n-\Big\lceil\frac{a}{2}\Big\rceil-1\Big)&\geq& e(\overline{G'})-\Big(n-\Big\lceil\frac{a}{2}\Big\rceil-1\Big)\\
&\geq&(q-1)(n-s-(q-1))+\frac{(q-1)(q-2)}{2}-\Big(n-\frac{a}{2}-1\Big)\\
&=& n-5>0,
\end{eqnarray*}
which implies that $e(\overline{G})>n-\Big\lceil\frac{a}{2}\Big\rceil-1,$ a contradiction.

Now we are in the position to consider the case of $t=1.$

\vspace{1.5mm}
\noindent\textbf{Case 2.2.} $t=1.$
\vspace{1.5mm}

Let $T=\{x_0\}.$ Then we have
\begin{eqnarray}\label{eq8}
 \theta_G(S,T)=a(s-1)+d_{G-S}(x_0)-o_G(S,T)<\varepsilon(S,T).
\end{eqnarray}

\vspace{1.5mm}
\noindent\textbf{Case 2.2.1.} $s=0.$
\vspace{1.5mm}

\begin{claim}\label{cla3}
$q=1.$
\end{claim}
\begin{proof}
Assume to the contrary that $q\geq 2.$ Let $C_1$ and $C_2$ be two components of $T'.$
Without loss of generality, let $|V(C_1)|\leq|V(C_2)|.$
We claim that $q=2,$ $|V(C_1)| = 1$ and $|V(C_2)| =n-2.$
Otherwise, by $n\geq 3a+4,$ we have
\begin{eqnarray*}
e(\overline{G}) \geq |V(C_1)||V(C_2)| \geq 2(n-3) > n-\Big\lceil\frac{a}{2}\Big\rceil-1,
\end{eqnarray*}
which contradicts $e(\overline{G})\leq n-\lceil\frac{a}{2}\rceil-1$.
Hence we have
\begin{eqnarray}\label{eq9}
e(\overline{G})\geq e(\overline{G}[T'])\geq n-2.
\end{eqnarray}
If $a\geq 3,$ then $e(\overline{G})\leq n-\lceil\frac{a}{2}\rceil-1=n-3,$ which contradicts (\ref{eq9}).
If $a=2,$ then $e(\overline{G})\leq n-2.$ By (\ref{eq9}), $e(\overline{G}) = e(\overline{G}[T']) = n-2.$ Hence
$d_{G-S}(x_0) = n-1.$ Note that $o_G(S,T)\leq q= 2.$
Combining $n\geq 3a+4$ and $a\geq 2,$ we have
\begin{eqnarray*}
 \theta_G(S,T)=-a+d_{G-S}(x_0)-o_G(S,T)\geq n-a-3\geq 2a+1\geq 5> \varepsilon(S,T),
\end{eqnarray*}
which contradicts (\ref{eq8}).
\end{proof}

\begin{claim}\label{cla4}
$\varepsilon(S,T) = 0.$
\end{claim}
\begin{proof}
Suppose to the contrary that $\varepsilon(S,T) =2.$ By Claim \ref{cla3}, $q=1.$ Note that $s=0.$ Then $S$ is independent.
Hence there exists a cut edge $e$ in $T'.$
Let $C'_1$ and $C'_2$ be two components of $T'-e.$
Without loss of generality, let $|V(C'_1)|\leq|V(C'_2)|.$
We claim that $|V(C'_1)| = 1$ and $|V(C'_2)| = n-2.$
Otherwise, combining $n\geq 3a+4$ and $a\geq 2,$ we can obtain that
\begin{eqnarray*}
e(\overline{G})\geq e(\overline{G}[T'])\geq |V(C'_1)||V(C'_2)|-1\geq 2(n-3)-1> n-\Big\lceil\frac{a}{2}\Big\rceil-1,
\end{eqnarray*}
contrary to $e(\overline{G})\leq n-\lceil\frac{a}{2}\rceil-1.$
Then $|V(C'_1)| = 1$ and $|V(C'_2)| = n-2,$ and hence $e(\overline{G}[V(T')])\geq n-3.$ If $a\geq 5,$ then $e(\overline{G})\geq e(\overline{G}[V(T')])\geq n-3>n-\lceil\frac{a}{2}\rceil-1,$ a contradiction.
If $2\leq a\leq 4,$ then we have
\begin{eqnarray}\label{eq10}
d_{G-S}(x_0)\geq n-1-(e(\overline{G})-e(\overline{G}[V(T')]))\geq n-2.
\end{eqnarray}
Combining (\ref{eq10}), Claim \ref{cla3}, $n\geq 3a+4$ and $2\leq a\leq 4,$ we have
\begin{eqnarray*}
 \theta_G(S,T)=-a+d_{G-S}(x_0)-o_G(S,T)\geq n-a-3\geq 2a+1> \varepsilon(S,T),
\end{eqnarray*}
which contradicts (\ref{eq8}).
\end{proof}

Recall that $\delta(G)\geq a.$ By Claim \ref{cla3}, $q=1.$ If $d_{G}(x_0) = a,$ then $e_{G}(T, V(T'))+a|V(T')| = d_{G}(x_0)+a(n-1) = na \equiv 0 \pmod{2},$ which implies that $T'$ is even. Hence $o_G(S,T) = 0.$
According to Claims \ref{cla3} and \ref{cla4}, we have $\theta_G(S,T)=0=\varepsilon(S,T),$ contrary to (\ref{eq8}).
If $d_G(x_0) \geq a+1,$ then $\theta_G(S,T)\geq 1-o_G(S,T)\geq 0 = \varepsilon(S,T),$ which contradicts (\ref{eq8}).

\vspace{1.5mm}
\noindent\textbf{Case 2.2.2.} $s \geq 1.$
\vspace{1.5mm}

We first prove the following claim.
\begin{claim}\label{cla5}
$q\geq 2.$
\end{claim}
\begin{proof}
Suppose to the contrary that $q=1.$
Then $e_{G}(T,V(T'))+at'=d_{G-S}(x_0)+a(n-s-1)=d_{G-S}(x_0)+a(s-1)+a(n-2s).$ Since $na\equiv 0\pmod{2},$ it follows that $a(n-2s)\equiv 0\pmod{2}.$ Hence we have
\begin{eqnarray}\label{eq11}
e_{G}(T,V(T'))+at'\equiv d_{G-S}(x_0)+a(s-1)\pmod{2}.
\end{eqnarray}
By $\delta(G)\geq a,$ $d_{G-S}(x_0)\geq a-s.$ For $a\geq 2$ and $s\geq 1,$ we have
\begin{eqnarray}\label{eq12}
 d_{G-S}(x_0)+a(s-1)\geq s(a-1)\geq 1.
\end{eqnarray}
If $T'$ is an odd component, then $o_G(S,T)=1$ and $\varepsilon(S,T)=0.$ By (\ref{eq12}), we have
\begin{eqnarray*}
 \theta_G(S,T)=a(s-1)+d_{G-S}(x_0)-o_G(S,T)\geq 0\geq \varepsilon(S,T),
\end{eqnarray*}
contradicting (\ref{eq8}).
If $T'$ is an even component, then $e_{G}(T,V(T'))+at'\equiv 0\pmod{2}$ and $o_G(S,T)=0.$ Combining (\ref{eq11}) and (\ref{eq12}), we have $a(s-1)+d_{G-S}(x_0)\geq 2.$ Hence we have
\begin{eqnarray*}
 \theta_G(S,T)=a(s-1)+d_{G-S}(x_0)-o_G(S,T)\geq 2\geq \varepsilon(S,T),
\end{eqnarray*}
contrary to (\ref{eq8}).
\end{proof}

By Claim \ref{cla5}, $T'$ has at least two components, which implies that
$e(\overline{G}[V(T')]) \geq n-s-2.$
Recall that $e(\overline{G})\leq n-\lceil\frac{a}{2}\rceil-1.$ Then we have
\begin{eqnarray}\label{eq13}
d_{G-S}(x_0)\geq n-s-1-\Big(\Big(n-\Big\lceil\frac{a}{2}\Big\rceil-1\Big)-(n-s-2)\Big)=n-2s+\Big\lceil\frac{a}{2}\Big\rceil-2.
\end{eqnarray}
Suppose that $o_G(S,T)\geq n-4.$
Then $e(\overline{G})\geq \frac{o_G(S,T)(o_G(S,T)-1)}{2}\geq \frac{(n-4)(n-5)}{2}.$
By $a\geq 2$ and $n\geq 3a+4\geq 10,$ we have
\begin{eqnarray*}
e(\overline{G})-\Big(n-\Big\lceil\frac{a}{2}\Big\rceil-1\Big)&\geq& \frac{(n-4)(n-5)}{2}-\Big(n-\Big\lceil\frac{a}{2}\Big\rceil-1\Big)\\
&\geq& \frac{(n-4)(n-5)}{2}-n+\frac{a}{2}+1\\
&\geq& \frac{(n-4)(n-5)}{2}-n+2\\
&\geq& \frac{1}{2}(n-8)(n-3) > 0,
\end{eqnarray*}
which contradicts $e(\overline{G})\leq n-\lceil\frac{a}{2}\rceil-1$.
Hence $o_G(S,T)\leq n-5.$
Combining (\ref{eq13}), $s\geq 1$ and $a\geq 2,$ we have
\begin{eqnarray*}
\theta_G(S,T)&=& a(s-1)+d_{G-S}(x_0)-o_G(S,T)\\
&\geq& a(s-1)+n-2s+\Big\lceil\frac{a}{2}\Big\rceil-2-(n-5)\\
&\geq& \frac{a}{2}+1 \geq 2\geq \varepsilon(S,T),
\end{eqnarray*}
which contradicts (\ref{eq8}).

\vspace{1.5mm}
\noindent\textbf{Case 2.3.} $t\geq 2.$
\vspace{1.5mm}

In this case, we have
\begin{eqnarray}\label{eq14}
\theta_G(S,T) = as-at+\sum_{x\in T} d_{G-S}(x)-o_G(S,T) < \varepsilon(S,T).
\end{eqnarray}
Let $C_1, C_2, \ldots, C_q$ be the $q$ components of $T'.$ Next we obtain the following claims.

\begin{claim}\label{cla6}
$q\leq 2.$
\end{claim}
\begin{proof}
Suppose to the contrary that $q\geq 3.$ By $e(\overline{G})\leq n-\lceil\frac{a}{2}\rceil-1,$ $e(\overline{G}[T])+e_{\overline{G}}(T,V(T')) \leq (n-\lceil\frac{a}{2}\rceil-1)-\frac{1}{2}\sum^q_{i=1}|C_i|(n-s-t-|C_i|).$
Hence we have
\begin{eqnarray}\label{eq15}
\sum_{x\in T} d_{G-S}(x)&\geq& (n-s-1)t-2\Big(n-\Big\lceil\frac{a}{2}\Big\rceil-\frac{1}{2}\sum^q_{i=1}|C_i|(n-s-t-|C_i|)-1\Big)\nonumber\\
&\geq & (n-s-1)t-2n+2\Big\lceil\frac{a}{2}\Big\rceil+q(n-s-t)-\sum^q_{i=1}|C_i|+2\nonumber\\
&\geq& (n-s-1)t-2n+a+(q-1)(n-s-t)+2.
\end{eqnarray}
Combining $(\ref{eq15}),$ $n-s-t\geq q\geq 3$ and $a\geq 2,$ we have
\begin{eqnarray}\label{eq16}
\theta_G(S,T) &=& as-at+\sum_{x\in T} d_{G-S}(x)-o_G(S,T)\nonumber\\
&\geq& as-at+(n-s-1)t-2n+a+(q-1)(n-s-t)-q+2\nonumber\\
&=& q(n-s-t-1)+a(s-t+1)+tn-3n-st+s+2\nonumber\\
&\geq& 3(n-s-t-1)+a(s-t+1)+tn-3n-st+s+2\nonumber\\
&=& a(s-t+1)+t(n-s-3)-2s-1.
\end{eqnarray}
Suppose that $s\leq \frac{n-1}{2}.$ Combining (\ref{eq16}), $n\geq 3a+4,$ $t\geq 2$ and $a\geq 2,$ we have
\begin{eqnarray*}
\theta_G(S,T) &\geq& t(n-s-a-3)+s(a-2)+a-1\\
&\geq& t\Big(\frac{n+1}{2}-a-3\Big)+s(a-2)+a-1\\
&\geq& \frac{t(a-1)}{2}+s(a-2)+a-1\\
&\geq& 2(a-1)+s(a-2)\\
&\geq& 2\geq \varepsilon(S,T),
\end{eqnarray*}
which contradicts (\ref{eq14}).
Next we assume that $s\geq \frac{n}{2}.$ Then $s\geq t+q.$ Combining $a\geq 2,$ $o_G(S,T)\leq q$ and $q\geq 3,$ we have
\begin{eqnarray*}
\theta_G(S,T) &=& as-at+\sum_{x\in T} d_{G-S}(x)-o_G(S,T)\\
&\geq& aq-o_G(S,T)\\
&\geq& q\geq 3> \varepsilon(S,T),
\end{eqnarray*}
contradicting (\ref{eq14}).
\end{proof}

\begin{claim}\label{cla7}
$s\geq t+2.$
\end{claim}
\begin{proof}
Suppose to the contrary that $s\leq t+1.$ Since $e(\overline{G})\leq n-\lceil\frac{a}{2}\rceil-1,$ $2e(\overline{G}[T])+e_{\overline{G}}(T,V(T')) \leq t(t-1)+\Big(n-\lceil\frac{a}{2}\rceil-1-\frac{t(t-1)}{2}\Big).$
Then
\begin{eqnarray}\label{eq17}
\sum_{x\in T} d_{G-S}(x)&\geq& t(n-s-1)-t(t-1)-\Big(n-\Big\lceil\frac{a}{2}\Big\rceil-\frac{t(t-1)}{2}-1\Big)\nonumber\\
&=& n(t - 1) - ts - \frac{t^2}{2} - \frac{t}{2} + \Big\lceil\frac{a}{2}\Big\rceil+1.
\end{eqnarray}
Combining (\ref{eq17}), $q\leq 2$ and $n\geq 3a+4,$ we have
\begin{eqnarray*}
\theta_G(S,T) &=& as-at+\sum_{x\in T} d_{G-S}(x)-o_G(S,T)\\
&\geq& as-at+n(t-1)-ts-\frac{t^2}{2}-\frac{t}{2}+\Big\lceil\frac{a}{2}\Big\rceil-q+1\\
&\geq& as-at+(3a+4)(t-1)-ts-\frac{t^2}{2}-\frac{t}{2}+\Big\lceil\frac{a}{2}\Big\rceil-1\\
&\geq& -\frac{t^2}{2}+\Big(2a-s+\frac{7}{2}\Big)t+as-\frac{5a}{2}-5.
\end{eqnarray*}
Let $f(t)=-\frac{t^2}{2}+\Big(2a-s+\frac{7}{2}\Big)t+as-\frac{5a}{2}-5.$ Recall that $G$ is not $[a,b]$-covered, $a=b\geq 2$ and $s\leq t+1.$
Combining Lemma \ref{lem3.1} and $t\geq 2,$ we have $2\leq t \leq 4.$
Hence $f(t) \geq \min\{f(2),f(4)\}.$
By $a\geq 2$ and $s\leq t+1\leq 5,$ we have $f(2)=(a-2)s+\frac{3a}{2}\geq 3$ and $f(4)=a(s+\frac{11}{2})-4s+1\geq 2.$
Hence $f(t)\geq \min\{f(2),f(4)\}\geq 2$ on $2\leq t\leq 4.$ Then $\theta_G(S,T)\geq 2\geq\varepsilon(S,T),$ contrary to (\ref{eq14}).
\end{proof}
Combining $a\geq 2,$ $o_G(S,T)\leq q,$ Claims \ref{cla7} and \ref{cla6}, we have
\begin{eqnarray*}
\theta_G(S,T) &=& as-at+\sum_{x\in T} d_{G-S}(x)-o_G(S,T) \\
&\geq& 4-q \geq 2 \geq \varepsilon(S,T),
\end{eqnarray*}
which contradicts (\ref{eq14}).

(ii) Now we are in the position to consider $b\neq 1,$ $a=1$ and $e(\overline{G})\leq n-1.$ Suppose to the contrary that $G$ is not $[a,b]$-covered. Then we have
\begin{eqnarray}\label{eq18}
\theta_G(S,T) = bs-t+\sum_{x\in T} d_{G-S}(x) < \varepsilon(S,T),
\end{eqnarray}
where
\begin{align*}
\begin{split}
\varepsilon(S,T)= \left \{
\begin{array}{ll}
    2,                    & S \mbox{ is not independent},\\
    1,                    & \mbox{$S$ is independent and there is a component $C$ of $G-S-T$ such}\\
                          & \mbox{that there exists an edge joining $C$ and $S.$}\\
    0,                    & \mbox{otherwise}.
\end{array}
\right.
\end{split}
\end{align*}
We claim that $s\leq t.$ Otherwise, $\theta_G(S,T) = bs-t+\sum_{x\in T} d_{G-S}(x) \geq b\geq 2\geq \varepsilon(S,T),$ which contradicts (\ref{eq18}). Note that $\delta(G)\geq 1.$
If $t\leq 1$ and $s=0,$ then $\varepsilon(S,T)=0$ and $\sum_{x\in T} d_{G-S}(x)\geq t.$ Hence $\theta_G(S,T)=bs-t+\sum_{x\in T} d_{G-S}(x)\geq 0=\varepsilon(S,T),$ contradicting (\ref{eq18}).
If $t=1$ and $s=1,$ then $\varepsilon(S,T)\leq 1.$ Thus $\theta_G(S,T)=bs-t+\sum_{x\in T} d_{G-S}(x)\geq b-1\geq 1\geq \varepsilon(S,T),$ contrary to (\ref{eq18}).
If $t=2,$ then we suppose that $u_1,u_2\in V(T).$ Combining $e(\overline{G})\leq n-1$ and $n\geq 7,$ we have
\begin{eqnarray*}
bs+\sum_{x\in T} d_{G-S}(x)&\geq& d_{G}(u_1)+d_{G}(u_2)\\
&\geq& 2(n-1)-2e(\overline{G}[T])-e_{\overline{G}}(T,V(T'))\\
&\geq& 2(n-1)-2-(n-2)=n-2\geq 5,
\end{eqnarray*}
which implies that $\theta_G(S,T)=bs-t+\sum_{x\in T} d_{G-S}(x)\geq 3>\varepsilon(S,T),$ which contradicts (\ref{eq18}).
If $t\geq 3,$ then we assume that $u_1,u_2,u_3\in V(T).$ By $e(\overline{G})\leq n-1$ and $n\geq 7,$ we have
\begin{eqnarray*}
(b+1)s+\sum_{x\in T} d_{G-S}(x)&\geq& d_{G}(u_1)+d_{G}(u_2)+d_{G}(u_3)\\
&\geq& 3(n-1)-6-(n-4)\\
&=& 2n-5\geq n+2.
\end{eqnarray*}
Then $\theta_G(S,T)=bs-t+\sum_{x\in T} d_{G-S}(x)\geq n-s-t+2\geq 2\geq \varepsilon(S,T),$ contradicting (\ref{eq18}).
\end{proof}

\vspace{3mm}

\noindent  \textbf{Proof of Theorem \ref{main0}.}
(i) $a \leq b,$ $b \geq 2$ and $n \geq 3a+4.$
Note that $\delta(H_{n,a})=a-1.$ Then $H_{n,a}$ has no $[a,b]$-factor. Hence $H_{n,a}$ is not $[a, b]$-covered. Suppose to the contrary that $G$ is not $[a,b]$-covered and $G\ncong H_{n,a}.$
We claim that $\delta(G)\geq a.$ Suppose to the contrary that $\delta (G) \leq a-1.$ Then $G$ is a subgraph of $H_{n,a},$ which implies that $G \cong H_{n,a}$ or $e(G) < e(H_{n,a}),$ a contradiction.
Note that $e(G)\geq e(H_{n,a})$ and $n\geq 3a+4.$ Hence $e(\overline{G})\leq n-a.$
If $a=1,$ then $e(\overline{G})\leq n-1.$
If $a\geq 2,$ then $e(\overline{G})\leq n-a\leq n-\lceil\frac{a}{2}\rceil-1.$
By $\delta(G)\geq a$ and Lemma \ref{lem3.3}, $G$ is $[a,b]$-covered, a contradiction.

\vspace{2mm}

(ii) $a=b=1$ and $n \geq 4.$
Note that $H_{n,3}=K_2\vee(K_1\cup K_{n-3})$ and $na=n \equiv 0\pmod{2}.$
Let $S=V(K_2),$ $T=V(K_1)$ and $T'=V(K_{n-3}).$
Hence $e(T,T')=0$ and $\sum_{x\in T} d_{G-S}(x)+t' = n-3 \equiv 1\pmod{2}.$
Then $T'$ is an odd component. Therefore, $o_G(S,T)=1.$
Since $S$ is not independent, $\varepsilon(S,T)=2.$
Then $\theta_G(S,T) = s-t-1=0<\varepsilon(S,T).$
By Theorem \ref{th1.4}, $H_{n,3}$ is not matching-covered.
It is easy to verify that $K_3\vee 3K_1$ is not matching-covered.
Suppose that $G$ is not matching-covered. Then there exists an edge $uv$ such that $G'=G-\{u,v\}$ contains no perfect matching.
By Theorem \ref{th1.2}, $e(G')\leq \binom{n-3}{2}$ with equality if and only if $G'\cong K_{n-3} \cup K_1$ or $K_{1,3}.$ We claim that $G'\cong K_{n-3} \cup K_1$ or $K_{1,3}.$ Otherwise, we have
$e(G')<\binom{n-3}{2}.$ Then $e(G)< \binom{n-3}{2}+2(n-2)+1=\binom{n-1}{2}+2,$ which contradicts $e(G)\geq \binom{n-1}{2}+2$. Hence $G$ is a subgraph of $H_{n,3}$ or $3K_1\vee K_3.$
By $e(G)\geq \binom{n-1}{2}+2,$ $G\cong H_{n,3}$ or $3K_1\vee K_3.$
\hspace*{\fill}$\Box$

\vspace{3mm}

\noindent  \textbf{Proof of Theorem \ref{main1}.}
(i) $a \leq b,$ $b \geq 2$ and $n \geq 3a+4.$ Recall that $H_{n,a}$ is not $[a, b]$-covered. Suppose to the contrary that $G$ is not $[a,b]$-covered and $G\ncong H_{n,a}.$
Note that $K_{n-1}$ is a subgraph of $H_{n,a}.$ Then we have
\begin{eqnarray}\label{eq19}
\rho(G) \geq \rho(H_{n,a}) \geq \rho(K_{n-1}) = n-2.
\end{eqnarray}
We claim that $\delta(G)\geq a.$ Suppose to the contrary that $\delta (G) \leq a-1.$ Then $G$ is a subgraph of $H_{n,a},$ which implies that $G \cong H_{n,a}$ or $\rho(G) < \rho(H_{n,a}),$ a contradiction.
By $\delta(G)\geq a,$ (\ref{eq19}) and Lemma \ref{lem2.2}, $e(\overline{G})\leq n-\lceil\frac{a}{2}\rceil-1.$
Combining $\delta(G)\geq a,$ $n\geq 3a+4,$ $b\neq 1$ and Lemma \ref{lem3.3} (i), we have $G$ is $[a,b]$-covered, a contradiction.

\vspace{2mm}

(ii) $a=b=1$ and $n \geq 4.$
Recall that $H_{n,3}$ is not matching-covered.
Suppose to the contrary that $G$ is not matching-covered and $G\ncong H_{n,3}.$ Note that $K_{n-1}$ is a subgraph of $H_{n,3}.$
Then $\rho(G) \geq \rho(H_{n,3}) \geq \rho(K_{n-1}) = n-2.$
We claim that $\delta(G)\geq 3.$ Suppose to the contrary that $\delta (G) \leq 2.$ Then $G$ is a subgraph of $H_{n,3},$ which implies that $G \cong H_{n,3}$ or $\rho(G) < \rho(H_{n,3}),$ a contradiction.
Combining $\delta(G)\geq 3,$ $\rho(G) \geq n-2$ and Lemma \ref{lem2.2}, we have $e(\overline{G})\leq n-3.$
By Theorem \ref{main0} and $G\ncong H_{n,3},$ we have $G\cong K_3\vee 3K_1.$
For $n=6,$ by equitable quotient matrices, one can obtain that $\rho(H_{6,3})> \rho(K_3\vee 3K_1),$ which contradicts $\rho(G) \geq \rho(H_{n,3}).$
\hspace*{\fill}$\Box$

\section{Concluding remarks}

In this paper, we provide complete characterization of the extremal graphs that maximize the size or the spectral radius within the family of non-$[a,b]$-covered graphs. Note that the extremal graphs in our main theorems contain large cliques. A natural and interesting direction for future research is to consider the extremal problems restricted to bipartite graphs. Since bipartite graphs contain no odd cycles, the clique-based extremal structures are forbidden. Consequently, the extremal bipartite graphs that are not $[a,b]$-covered will be fundamentally different. Motivated by this, we propose the following problems.

\begin{prob}
Determine the maximum size $e(G)$ of an $n$-vertex bipartite graph $G$ that is not $[a,b]$-covered, and characterize the corresponding size-extremal graphs.
\end{prob}

\begin{prob}
Determine the maximum spectral radius $\rho(G)$ of an $n$-vertex bipartite graph $G$ that is not $[a,b]$-covered, and characterize the spectral extremal graphs.
\end{prob}

\vspace{5mm}
\noindent
{\bf Declaration of competing interest}
\vspace{3mm}

The authors declare that they have no known competing financial interests or personal relationships that could have appeared to influence the work reported in this paper.

\vspace{5mm}
\noindent
{\bf Data availability}
\vspace{3mm}

No data was used for the research described in this paper.



\end{document}